\documentclass{llncs}
\usepackage{makeidx}  % allows for indexgeneration
\usepackage{mathtools}
\usepackage{amsmath,amssymb}
\usepackage{algorithm,algpseudocode}

\newcommand{\norm}[1]{\left\lVert #1 \right\rVert} % norm
\DeclareMathOperator*{\argmin}{argmin}

\begin{document}
\frontmatter          % for the preliminaries
\pagestyle{headings}  % switches on printing of running heads
\addtocmark{Primal-dual numerical method} % additional mark in the TOC
\mainmatter              % start of the contributions
\title{Primal-Dual Method for Searching Equilibrium\\
in Hierarchical Congestion Population Games}
\titlerunning{Searching Equilibrium
in Hierarchical Population Games}  % abbreviated title (for running head)
\author{Pavel Dvurechensky\inst{1} \and Alexander Gasnikov\inst{2} \and Evgenia Gasnikova\inst{3}
\and Sergey Matsievsky\inst{4} \and Anton Rodomanov \inst{5} \and Inna Usik\inst{6}}
\authorrunning{Pavel Dvurechensky et al.} % abbreviated author list (for running head)
\tocauthor{Pavel Dvurechensky, Alexander Gasnikov, Evgenia Gasnikova, Sergey Matsievsky, Anton Rodomanov,
and Inna Usik}
\institute{
Weierstrass Institute for Applied Analysis and Stochastics, Berlin 10117, Germany, \\
Institute for Information Transmission Problems, Moscow 127051, Russia,
\email{pavel.dvurechensky@wias-berlin.de}
\and
Moscow Institute of Physics and Technology,
\\Dolgoprudnyi 141700, Moscow Oblast, Russia,\\
Institute for Information Transmission Problems, Moscow 127051, Russia, \\
\email{gasnikov@yandex.ru} 
\and Moscow Institute of Physics and
Technology, \\Dolgoprudnyi 141700, Moscow Oblast, Russia,\\
\email{egasnikova@yandex.ru} 
\and Immanuel Kant Baltic
Federal University, Kaliningrad 236041, Russia,\\
\email{matsievsky@newmail.ru} 
\and Higher School of Economics National Research University, Moscow 125319, Russia \\
	\email{anton.rodomanov@gmail.com}
\and Immanuel Kant Baltic
Federal University, Kaliningrad 236041, Russia,\\
\email{lavinija@mail.ru}
}

\maketitle              % typeset the title of the contribution

\begin{abstract}
In this paper, we consider a large class of hierarchical congestion population games.
One can show that the equilibrium in a game of such type can be
described as a minimum point in a properly constructed multi-level convex optimization
problem. 
We propose a fast primal-dual composite gradient method and apply it to the problem, which is dual to the problem describing the equilibrium in the considered class of games.
We prove that this method allows to find an approximate solution of the initial
problem without increasing the complexity.

\end{abstract}

\section{Problem Statement}
In this subsection, we briefly describe a variational principle for equilibrium description in hierarchical congestion population games. In particular, we consider a multistage model of traffic flows. Further details can be found in \cite{GasGasMatsUs2015}.

We consider the traffic network described by the directed graph
$\Gamma ^1=\left\langle {V^1,E^1} \right\rangle $. 
Some of its vertices $O^1\subseteq V^1$ are sources (origins), and some are
sinks (destinations) $D^1\subseteq V^1$. 
We denote a set of source-sink pairs by $OD^1\subseteq O^1\otimes D^1$. 
Let us assume that for each pair $w^1\in OD^1$ there is a flow of network users of the amount of $d_{w^1}^1 :=d_{w^1}^1 \cdot M$
, where $M\gg 1$, per unit time who moves from the origin of $w^1$ to its destination. We call the pair $w^1, d_{w^1}^1$ as correspondence.

Let edges $\Gamma ^1$ be partitioned into two types $E^1=\tilde
{E}^1\coprod \bar {E}^1$. 
The edges of type $\tilde {E}^1$ are
characterized by non-decreasing functions of expenses $\tau _{e^1}^1
(f_{e^1}^1) := \tau _{e^1}^1 (f_{e^1}^1/M)$. 
Expenses $\tau _{e^1}^1
(f_{e^1}^1)$ are incurred by those users who use in their path an edge
$e^1\in \tilde {E}^1$, the flow of users on
this edge being equal to $f_{e^1}^1 $. 
The pairs of vertices setting the
edges of type $\bar {E}^1$ are in turn a source-sink pairs
$OD^2$ (with correspondences $d_{w^2}^2 =f_{e^1}^1 $, $w^2=e^1\in
\bar {E}_1$) in a traffic network of the second level $\Gamma
^2=\left\langle {V^2,E^2} \right\rangle$ whose edges are partitioned
in turn into two types $E^2=\tilde {E}^2\coprod \bar {E}^2$.  
The
edges having type $\tilde {E}^2$ are characterized by non-decreasing
functions of expenses $\tau _{e^2}^2 (f_{e^2}^2) := \tau _{e^2}^2
({f_{e^2}^2/M})$. 
Expenses $\tau _{e^2}^2 (f_{e^2}^2)$ are incurred
by those users who use in their path an edge $e^2\in \tilde {E}^2$, the flow of users on this edge being equal to
$f_{e^2}^2 $.

The pairs of vertices setting the edges having type $\bar {E}^2$ are in
turn source-sink pairs $OD^3$ (with correspondences $d_{w^3}^3
=f_{e^2}^2 $, $w^3=e^2\in \bar {E}^2$) in a traffic network of a
higher level $\Gamma ^3=\left\langle {V^3,E^3} \right\rangle $, etc.
We assume that in total there are $m$ levels: $\tilde
{E}^m=E^m$. Usually, in applications, the number $m$ is small and varies from 2 to 10.

Let $P_{w^1}^1 $ be the set of all paths in $\Gamma ^1$ which correspond to a correspondence
$w^1$. Each user in the graph $\Gamma ^1$ chooses a path $p_{w^1}^1 \in
P_{w^1}^1 $ (a consecutive set of the edges passed by the user)
corresponding to his correspondence $w^1\in OD^1$. 
Having defined a path $p_{w^1}^1$, it is possible to restore
unambiguously the edges having type $\bar {E}^1$ which belong to this path.
On each of these edges $w^2\in \bar {E}^1$, user can choose a path
$p_{w^2}^2 \in P_{w^2}^2$ ($P_{w^2}^2$ is a set of all paths
corresponding in the graph $\Gamma ^2$ to the correspondence $w^2$), etc. Let us assume that each
user have made the choice.

We denote by $x_{p^1}^1 $ the size of the flow of users on
a path $p^1\in P^1=\coprod\limits_{w^1\in OD^1} {P_{w^1}^1 }$,
$x_{p^2}^2 $ the size of the flow of users on a path $p^2\in
P^2=\coprod\limits_{w^2\in OD^2} {P_{w^2}^2 }$, etc. Let us notice
that
\[
x_{p_{w^k}^k }^k \ge 0, \quad p_{w^k}^k \in P_{w^k}^k , \quad
\sum\limits_{p_{w^k}^k \in P_{w^k}^k } {x_{p_{w^k}^k }^k }
=d_{w^k}^k , \quad w^k\in OD^k, \quad k=1,...,m
\]
and that
\[
w^{k+1}\left( {=e^k} \right)\in OD^{k+1}\left( {=\bar {E}^k}
\right), \quad d_{w^{k+1}}^{k+1} =f_{e^k}^k , \quad k=1,...,m-1.
\]

For all $k=1,...,m$, we introduce for the graph $\Gamma ^k$ and the set of paths $P^k$ a matrix
\[
\Theta ^k=\left\| {\delta _{e^kp^k} } \right\|_{e^k\in E^k,p^k\in
P^k} , \quad \delta _{e^kp^k} =\left\{ {\begin{array}{l}
 1,\mbox{ }e^k\in p^k \\
 0,\mbox{ }e^k\notin p^k \\
 \end{array}} \right..
%\quad k=1,...,m.
\]
Then, for all $k=1,...,m$, the vector $f^k$ of flows on the edges of the graph $\Gamma ^k$ is defined in a unique way
by the vector of flows on the paths $x^k=\bigl\{
{x_{p^k}^k } \bigr\}_{p^k\in P^k}$:
\[
f^k=\Theta ^kx^k.
\]
We introduce the following notation	
\[
x=\left\{ {x^k} \right\}_{k=1}^m , \quad f=\left\{ {f^k}
\right\}_{k=1}^m , \quad \Theta =\mbox{diag}\left\{ {\Theta ^k}
\right\}_{k=1}^m .
\]

We denote $E = \coprod\limits_{k = 1}^m {{{\tilde E}^k}} $ and set $t = {\left\{ {{t_e}} \right\}_{e \in E}}$. Further, we define by induction, with the basis $g_{{p^m}}^m\left( t \right) = \sum\limits_{{e^m} \in {E^m}} {{\delta _{{e^m}{p^m}}}{t_{{e^m}}}} $, for all $k<m$, 
$$
g_{{p^k}}^k\left( t \right) = \sum\limits_{{e^k} \in {{\tilde E}^k}} {{\delta _{{e^k}{p^k}}}{t_{{e^k}}}}  - \sum\limits_{{e^k} \in {{\bar E}^k}} {{\delta _{{e^k}{p^k}}}{\gamma ^{k + 1}}\psi _{{e^k}}^{k + 1}\left( {{t \mathord{\left/
 {\vphantom {t {{\gamma ^{k + 1}}}}} \right.
 \kern-\nulldelimiterspace} {{\gamma ^{k + 1}}}}} \right)} 
$$ 
to be the "`length"' of the path $p^k$ in the graph $\Gamma^k$ with the edges ${e^k} \in {\tilde E^k}$ having weight ${t_{{e^k}}}$ and the edges ${e^k} \in {\bar E^k}$ having weight ${\gamma ^{k + 1}}\psi _{{e^k}}^{k + 1}\left( {{t \mathord{\left/
 {\vphantom {t {{\gamma ^{k + 1}}}}} \right.
 \kern-\nulldelimiterspace} {{\gamma ^{k + 1}}}}} \right)$. 
Here ${\gamma ^{k + 1}} \ge 0$ is the parameter, characterizing the restricted rationality of the network users on the level $k$, and 
$$
\psi _{{e^k}}^{k + 1}\left( t \right) = \psi _{{w^{k + 1}}}^{k + 1}\left( t \right) = \ln \left( {\sum\limits_{{p^{k + 1}} \in P_{{w^{k + 1}}}^{k + 1}} {\exp \left( { - g_{{p^{k + 1}}}^{k + 1}\left( t \right)} \right)} } \right).
$$  
Let us now describe the probabilistic model for the choice of the path by a network user. 
We assume that each user $l$ of a traffic network who uses a
correspondence $w^k\in OD^k$ at a level $k$ (and simultaniously the edge $e^{k-1}(=
w^k)\in \bar {E}^{k-1}$ at the level $k-1$) chooses to use a path $p^k\in
P_{w^k}^k$ if
\[
p^k=\arg \mathop {\max }\limits_{q^k\in P_{w^k}^k}\{{-g_{q^k}^k
(t)+\xi _{q^k}^{k,l}}\},
\]
where $\xi _{q^k}^{k,l}$ are iid random variables with double exponential distribution (also known as 
Gumbel's distribution) with cumulative distribution function
\[
P(\xi _{q^k}^{k,l} <\zeta)=\exp \{ {-e^{-\zeta/\gamma ^k - E}} \},
\]
where $E\approx 0.5772$ is Euler--Mascheroni constant. In this case
\[
M[ {\xi _{q^k}^{k,l} } ]=0, \quad D[ {\xi _{q^k}^{k,l} } ] = (
\gamma^k)^2\pi^2/6.
\]
Also, it turns out  that, when the number of agents on each
correspondence $w^k\in OD^k$, $k=1,...,m$ tends to infinity, i.~e. $M\to\infty$, the limiting distribution of users among paths is the Gibbs's distribution (also known as logit distribution)
\begin{equation}
\label{Gas_eq1} x_{p^k}^k =d_{w^k}^k \frac{\exp (-g_{p^k}^k
(t)/\gamma ^k )}{\sum\limits_{\tilde {p}^k\in P_{w^k}^k } {\exp (
-g_{\tilde {p}^k}^k (t)/\gamma ^k)} },  p^k\in P_{w^k}^k ,  w^k\in
OD^k, k=1,...,m.
\end{equation}
It is worth noting here that (see Theorem 1
below)
\begin{align}
&{\gamma ^k}\psi _{{w^k}}^k\left( {{t \mathord{\left/
 {\vphantom {t {{\gamma ^k}}}} \right.
 \kern-\nulldelimiterspace} {{\gamma ^k}}}} \right) = {E_{{{\left\{ {\xi _{{p^k}}^{k,l}} \right\}}_{{p^k} \in P_{{w^k}}^k}}}}\left[ {\mathop {\max }\limits_{{p^k} \in P_{{w^k}}^k} \left\{ { - g_{{p^k}}^k\left( t \right) + \xi _{{p^k}}^{k,l}} \right\}} \right]. \notag \\
& f = \Theta x =  - \nabla {\psi ^1}\left( {{t \mathord{\left/
 {\vphantom {t {{\gamma ^1}}}} \right.
 \kern-\nulldelimiterspace} {{\gamma ^1}}}} \right), {\psi ^1}\left( t \right) = \sum\limits_{{w^1} \in O{D^1}} {d_{{w^1}}^1\psi _{{w^1}}^1\left( t \right)}. \qquad \qquad \qquad \qquad \qquad \quad (1') \notag %\eqno{1'}
\end{align}

For the sake of convenience we introduce the graph
\[
\Gamma =\coprod\limits_{k=1}^m {\Gamma ^k} = \biggl\langle
{V,E=\coprod\limits_{k=1}^m {\tilde {E}^k} } \biggr\rangle
\]
and denote $t_e =\tau _e (f_e)$, $e \in E$.

Assume that, for a given vector of expenses $t$ on edges $E$, which is identical to all users, each user chooses the shortest path at each level based on noisy information and averaging of the information from the higher levels.
Then, in the limit number of users tending to infinity, such behavior of users leads to the description of
distribution of users on paths/edges given in (\ref{Gas_eq1}) and the
equilibrium configuration in the system is characterized by the vector $t$
for which the vector $x$, obtained from (\ref{Gas_eq1}), leads to the vector $f=\Theta x$ (see also $(1')$) satisfying $t = \{ \tau _e (f_e )
\}_{e\in E}$.

Introducing ${\sigma _e}\left( {{f_e}} \right) = \int\limits_0^{{f_e}} {{\tau _e}\left( z \right)dz} $ and $\sigma _e^*\left( {{t_e}} \right) = \mathop {\max }\limits_{{f_e}} \left\{ {{f_e}{t_e} - {\sigma _e}\left( {{f_e}} \right)} \right\}$, we obtain 
$$
\frac{{d\sigma _e^*\left( {{t_e}} \right)}}{{d{t_e}}} = \frac{d}{{d{t_e}}}\mathop {\max }\limits_{{f_e}} \left\{ {{f_e}{t_e} - \int\limits_0^{{f_e}} {{\tau _e}\left( z \right)dz} } \right\} = {f_e}: {t_e} = {\tau _e}\left( {{f_e}} \right), e \in E.
$$

This allows to prove the following.

\begin{theorem}[Variational principle] The fixed point equilibrium $x,f,t$ can be found as a solution of the following problem (here and below we denote by $\mbox{dom}\;\sigma_e^\ast$ the effective domain of
the function conjugated to a function $\sigma_e$)
\begin{equation}
\label{Gas_eq2} \mathop {\min }\limits_{f,x} \{ {\Psi
(x,f):\;f=\Theta x,\;x\in X} \} = - \mathop {\min }\limits_{t \in {{\left\{ {{\rm{dom}}\;\,\sigma _e^*} \right\}}_{e \in E}}} \left\{ {{\gamma ^1}{\psi ^1}\left( {{t \mathord{\left/
 {\vphantom {t {{\gamma ^1}}}} \right.
 \kern-\nulldelimiterspace} {{\gamma ^1}}}} \right) + \sum\limits_{e \in E} {\sigma _e^*\left( {{t_e}} \right)} } \right\},
\end{equation}
\vspace{-2ex}
\textit{where}
\[
\Psi (x,f):=\Psi ^1(x)=\sum\limits_{e^1\in \tilde {E}^1} {\sigma
_{e^1}^1 (f_{e^1}^1)} +\Psi ^2(x)+\gamma ^1\sum\limits_{w^1\in OD^1}
{\sum\limits_{p^1\in P_{w^1}^1 } {x_{p^1}^1 \ln
(x_{p^1}^1/d_{w^1}^1)} },
\]
\[
\Psi ^2(x)=\sum\limits_{e^2\in \tilde {E}^2} {\sigma _{e^2}^2
(f_{e^2}^2)} +\Psi ^3(x)+\gamma ^2\sum\limits_{w^2\in \bar {E}^1}
{\sum\limits_{p^2\in P_{w^2}^2 } {x_{p^2}^2 \ln
(x_{p^2}^2/d_{w^2}^2)} }, d_{w^2}^2 =f_{w^2}^1 ,
\]
\vspace{-5ex}
\begin{center}
\textit{\ldots}
\end{center}
\vspace{-2ex}
\[
\Psi ^k(x)=\sum\limits_{e^k\in \tilde {E}^k} {\sigma _{e^k}^k
(f_{e^k}^k)} +\Psi ^{k+1}(x)+\gamma ^k\sum\limits_{w^k\in \bar
{E}^{k-1}} {\sum\limits_{p^k\in P_{w^k}^k } {x_{p^k}^k \ln
(x_{p^k}^k/d_{w^k}^k)} },
\]
\vspace{-2ex}
\[
d_{w^{k+1}}^{k+1} =f_{w^{k+1}}^k ,
\]
\vspace{-6ex}
\begin{center}
\textit{\ldots}
\end{center}
\vspace{-2ex}
\[
\Psi ^m(x)=\sum\limits_{e^m\in E^m} {\sigma _{e^m}^m \left(
{f_{e^m}^m } \right)} +\gamma ^m\sum\limits_{w^m\in \bar {E}^{m-1}}
{\sum\limits_{p^m\in P_{w^m}^m } {x_{p^m}^m \ln
(x_{p^m}^m/d_{w^m}^m)} },
\]
\[
 d_{w^m}^m =f_{w^m}^{m-1}.
\]
%\[
%\sigma _e^\ast (t_e) = \mathop {\max }\limits_{f_e } \biggl\{ {f_e
%t_e -\int\limits_0^{f_e } {\tau _e (z)dz} } \biggr\},
%\]
%\vspace{-2ex}
%\[
%\frac{d\sigma _e^\ast (t_e)}{dt_e }=\frac{d}{dt_e }\mathop {\max
%}\limits_{f_e } \biggl\{ {f_e t_e -\int\limits_0^{f_e } {\tau _e
%(z)dz} } \biggr\}=f_e : \quad t_e =\tau _e (f_e), \quad e\in E,
%\]
%\[
%g_{p^m}^m (t)=\sum\limits_{e^m\in \tilde {E}^m} {\delta _{e^mp^m}
%t_{e^m} } =\sum\limits_{e^m\in E^m} {\delta _{e^mp^m} t_{e^m} } ,
%\]
%\[
%g_{p^k}^k (t)=\sum\limits_{e^k\in \tilde {E}^k} {\delta _{e^kp^k}
%t_{e^k} } -\sum\limits_{e^k\in \bar {E}^k} {\delta _{e^kp^k} \gamma
%^{k+1}\psi _{e^k}^{k+1} (t/\gamma ^{k+1})}, k=1,...,m-1,
%\]
%\[
%\psi _{w^k}^k (t) = \ln \biggl( {\sum\limits_{p^k\in P_{w^k}^k }
%{\exp (-g_{p^k}^k (t))} } \biggr), \quad k=1,...,m,
%\]
%\[
%\psi ^1 (t) = \sum\limits_{w^1\in OD^1} {d_{w^1}^1 \psi _{w^1}^1
%(t)}.
%\]
\end{theorem}

\section{General Numerical Method}
In this subsection, we describe one of our contributions made by this paper, namely a general accelerated primal-dual gradient method for composite minimization problems.

We consider the following convex composite optimization problem~\cite{Nes2013}:
\begin{equation}
\min_{x \in Q} \ [ \phi(x) := f(x) + \Psi(x) ].
\end{equation}
Here $Q \subseteq E$ is a closed convex set, the function $f$ is differentiable and convex on $Q$, and function $\Psi$ is closed and convex on $Q$ (not necessarily differentiable).

In what follows we assume that $f$ is $L_f$-smooth on $Q$:
\begin{equation}\label{eq:lipschitz}
\norm{ \nabla f(x) - \nabla f(y) }_* \leq L_f \norm{x - y}, \qquad \forall x, \, y \in Q.
\end{equation}
We stress that the constant $L_f > 0$ arises only in theoretical analysis and not in the actual implementation of the proposed method. Moreover, we assume that the set $Q$ is unbounded and that $L_f$ can be unbounded on the set $Q$.

The space $E$ is endowed with a norm $\norm{\cdot}$ (which can be arbitrary). The corresponding dual norm is $\norm{g}_* := \max_{x \in E} \{\langle g, x \rangle : \norm{x} \leq 1 \}, \ g \in E^*$.
For mirror descent, we need to introduce the Bregman divergence. Let $\omega: Q \rightarrow \mathbb{R}$ be a distance generating function, i.e. a 1-strongly convex function on $Q$ in the $\norm{\cdot}$-norm:
\begin{equation}
\omega(y) \geq \omega(x) + \langle \omega'(w), y - x \rangle + \frac{1}{2} \norm{y - x}^2, \qquad \forall x, \, y \in Q.
\end{equation}
Then, the corresponding Bregman divergence is defined as
\begin{equation}
V_x(y) := \omega(y) - \omega(x) - \langle \omega'(x), y - x \rangle, \qquad x, \, y \in Q.
\end{equation}

Finally, we generalize the $\mathrm{Grad}$ and $\mathrm{Mirr}$ operators from \cite{AllOre} to composite functions:
\begin{equation}
\begin{aligned}
\mathrm{Grad}_L(x) &:= \argmin_{y \in Q} \left\{ \langle \nabla f(x), y - x \rangle + \frac{L}{2} \norm{y - x}^2 + \Psi(y) \right\}, &&x \in Q, \\
\mathrm{Mirr}_z^{\alpha}(g) &:= \argmin_{y \in Q} \left\{ \langle g, y - z \rangle + \frac{1}{\alpha} V_z(y) + \Psi(y) \right\}, && g \in E^*, \ z \in Q.
\end{aligned}
\end{equation}

\subsection{Algorithm description}

Below is the proposed scheme of the new method. The main differences between this algorithm and the algorithm of \cite{AllOre} are as follows: 1) now the $\mathrm{Grad}$ and $\mathrm{Mirr}$ operators contain the $\Psi(y)$ term inside; 2) now the algorithm does not require the actual Lipschitz constant $L_f$, instead it requires an arbitrary number $L_0$\footnote{The number $L_0$ can be always set to 1 with virtually no harm to the convergence rate of the method.} and automatically adapts the Lipschitz constant in iterations; 3) now we need to use a different formula for $\alpha_{k+1}$ to guarantee convergence (see next section).

\begin{algorithm}
\begin{algorithmic}
\Require $x_0 \in Q$: initial point; $T$: number of iterations; $L_0$: initial estimate of $L_f$.
\State $y_0 \gets x_0$, $z_0 \gets x_0$, $\alpha_0 \gets 0$
\For{$k = 0, \dots, T-1$}
	\State $L_{k+1} \gets \max \{ L_0, L_k/2 \}$
	\While{True}
		\State $\alpha_{k+1} \gets \sqrt{\alpha_k^2 \frac{L_k}{L_{k+1}} + \frac{1}{4 L_{k+1}^2}} + \frac{1}{2 L_{k+1}}, \ \text{and} \ \tau_k \gets \frac{1}{\alpha_{k+1} L_{k+1}}.$
		\State $x_{k+1} \gets \tau_k z_k + (1 - \tau_k) y_k$
		\State $y_{k+1} \gets \mathrm{Grad}_{L_{k+1}}(x_{k+1})$
		\State{\textbf{if} $f(y_{k+1}) \leq f(x_{k+1}) + \langle \nabla f(x_{k+1}), y_{k+1} - x_{k+1} \rangle + \frac{L_{k+1}}{2} \norm{y_{k+1} - x_{k+1}}^2$ \textbf{then break}}
		\State $L_{k+1} \gets 2 L_{k+1}$
	\EndWhile
	\State $z_{k+1} \gets \mathrm{Mirr}_{z_k}^{\alpha_{k+1}}(\nabla f(x_{k+1}))$
\EndFor

\Return $y_T$
\end{algorithmic}
\caption{\label{alg:main}Accelerated gradient method.}
\end{algorithm}

Note that Algorihtm~\ref{alg:main} if well-defined in the sense that it is always guaranteed that $\tau_k \in [0, 1]$ and, hence, $x_{k+1} \in Q$ as a convex combination of points from $Q$. Indeed, from the formula for $\alpha_{k+1}$ we have
\begin{equation}
\alpha_{k+1} L_{k+1} \geq \left( \sqrt{\frac{1}{4 L_{k+1}^2}} + \frac{1}{2 L_{k+1}} \right) L_{k+1} = 1,
\end{equation}
therefore $\tau_k = \frac{1}{\alpha_{k+1} L_{k+1}} \leq 1$. 

\subsection{Convergence rate}

First we prove the analogues of Lemma 4.2 and Lemma 4.3 from \cite{AllOre}.

\begin{lemma}\label{lem:first_lemma} For any $u \in Q$ and $\tau_k = \frac{1}{\alpha_{k+1} L_{k+1}}$ we have
\begin{multline}
\alpha_{k+1} \langle \nabla f(x_{k+1}), z_k - u \rangle \leq \alpha_{k+1}^2 L_{k+1} ( \phi(x_{k+1}) - \phi(y_{k+1}) ) + (V_{z_k}(u) - V_{z_{k+1}}(u)) \\ + \alpha_{k+1} \Psi(u) - (\alpha_{k+1}^2 L_{k+1}) \Psi(x_{k+1}) + (\alpha_{k+1}^2 L_{k+1} - \alpha_{k+1}) \Psi(y_k).
\end{multline}
\end{lemma}
\begin{proof}
From the first order optimality condition for $z_{k+1} = \mathrm{Mirr}_{z_k}^{\alpha_{k+1}}(\nabla f(x_{k+1}))$ we get
\begin{equation}
\left\langle \nabla f(x_{k+1}) + \frac{1}{\alpha_k} V'_{z_k}(z_{k+1}) + \Psi'(z_{k+1}), z_{k+1} - u \right\rangle \leq 0, \qquad \forall u \in Q.
\end{equation}
Therefore
\begin{equation}
\begin{aligned}
\alpha_{k+1} \langle \nabla f(x_{k+1}), z_k - u \rangle \\
&\hspace{-7em}= \alpha_{k+1} \langle \nabla f(x_{k+1}), z_k - z_{k+1} \rangle + \alpha_{k+1} \langle \nabla f(x_{k+1}), z_{k+1} - u \rangle \\
&\hspace{-7em}\leq \alpha_{k+1} \langle \nabla f(x_{k+1}), z_k - z_{k+1} \rangle + \langle V_{z_k}'(z_{k+1}), u - z_{k+1} \rangle 
\\
&\hspace{-7em} + \alpha_{k+1} \langle \Psi'(z_{k+1}), u - z_{k+1} \rangle 
\\
&\hspace{-7em}\leq ( \alpha_{k+1} \langle \nabla f(x_{k+1}), z_k - z_{k+1} \rangle - \alpha_{k+1} \Psi(z_{k+1}) ) 
\\
&\hspace{-7em} + \langle V_{z_k}'(z_{k+1}), u - z_{k+1} \rangle + \alpha_{k+1} \Psi(u),
\end{aligned}
\end{equation}
where the second inequality follows from the convexity of $\Psi$.

Using the triangle equality of the Bregman divergence, 
$$
\langle V'_x(y), u - y \rangle~=~V_x(u) - V_y(u) - V_x(y),
$$
 we get
\begin{equation}
\begin{aligned}
\langle V_{z_k}'(z_{k+1}), u - z_{k+1} \rangle &= V_{z_k}(u) - V_{z_{k+1}}(u) - V_{z_k}(z_{k+1}) \\
&\leq V_{z_k}(u) - V_{z_{k+1}}(u) - \frac{1}{2} \norm{z_{k+1} - z_k}^2,
\end{aligned}
\end{equation}
where we have used $V_{z_k}(z_{k+1}) \geq \frac{1}{2} \norm{ z_{k+1} - z_k }^2$ in the last inequality.

So we have
\begin{equation}
\begin{aligned}
&\alpha_{k+1} \langle \nabla f(x_{k+1}), z_k - u \rangle  \\
& \leq \left( \alpha_{k+1} \langle \nabla f(x_{k+1}), z_k - z_{k+1} \rangle - \frac{1}{2} \norm{z_{k+1} - z_k}^2 - \alpha_{k+1} \Psi(z_{k+1}) \right) \\
&\quad + (V_{z_k}(u) - V_{z_{k+1}}(u)) + \alpha_{k+1} \Psi(u)
\end{aligned}
\end{equation}

Define $v := \tau_k z_{k+1} + (1 - \tau_k) y_k \in Q$. Then we have $x_{k+1} - v = \tau_k (z_k - z_{k+1})$ and $\tau_k \Psi(z_{k+1}) + (1 - \tau_k) \Psi(y_k) \geq \Psi(v)$ due to convexity of $\Psi$. Using this and the formula for $\tau_k$, we get
\begin{equation}
\begin{aligned}
&\left( \alpha_{k+1} \langle \nabla f(x_{k+1}), z_k - z_{k+1} \rangle - \frac{1}{2} \norm{z_{k+1} - z_k}^2 - \Psi(z_{k+1}) \right) \\
&\leq -\left( \frac{\alpha_{k+1}}{\tau_k} \langle \nabla f(x_{k+1}), v - x_{k+1} \rangle + \frac{1}{2 \tau_k^2} \norm{v - x_{k+1}}^2 + \frac{\alpha_{k+1}}{\tau_k} \Psi(v) \right) \\
& + \frac{\alpha_{k+1} (1 - \tau_k)}{\tau_k} \Psi(y_k) \\
&\leq -(\alpha_{k+1}^2 L_{k+1}) \left( \langle \nabla f(x_{k+1}), v - x_{k+1} \rangle + \frac{L_{k+1}}{2} \norm{v - x_{k+1}}^2 + \Psi(v) \right) \\
& + (\alpha_{k+1}^2 L_{k+1} - \alpha_{k+1}) \Psi(y_k) \\
&\leq -(\alpha_{k+1}^2 L_{k+1}) \left( \langle \nabla f(x_{k+1}), y_{k+1} - x_{k+1} \rangle + \frac{L_{k+1}}{2} \norm{y_{k+1} - x_{k+1}}^2 + \Psi(y_{k+1}) \right) \\
&+ (\alpha_{k+1}^2 L_{k+1} - \alpha_{k+1}) \Psi(y_k)
\end{aligned}
\end{equation}
Here the last inequality follows from the definition of $y_{k+1}$.

Note that by the termination condition for choosing $L_{k+1}$ we have
\begin{equation}
\begin{aligned}
\phi(y_{k+1}) &= f(y_{k+1}) + \Psi(y_{k+1}) \\
&\leq f(x_{k+1}) + \langle \nabla f(x_{k+1}), y_{k+1} - x_{k+1} \rangle \\
& + \frac{L_{k+1}}{2} \norm{y_{k+1} - x_{k+1}}^2 + \Psi(y_{k+1}) \\
&= \phi(x_{k+1}) + \langle \nabla f(x_{k+1}), y_{k+1} - x_{k+1} \rangle \\
& + \frac{L_{k+1}}{2} \norm{y_{k+1} - x_{k+1}}^2 + \Psi(y_{k+1}) - \Psi(x_{k+1}).
\end{aligned}
\end{equation}
After rearranging:
\begin{equation}
\begin{aligned}
& -\left( \langle \nabla f(x_{k+1}), y_{k+1} - x_{k+1} \rangle + \frac{L_{k+1}}{2} \norm{y_{k+1} - x_{k+1}}^2 + \Psi(y_{k+1}) \right) \\
& \leq \phi(x_{k+1}) - \phi(y_{k+1}) - \Psi(x_{k+1}).
\end{aligned}
\end{equation}
Hence,
\begin{equation}
\begin{aligned}
&\left( \alpha_{k+1} \langle \nabla f(x_{k+1}), z_k - z_{k+1} \rangle - \frac{1}{2} \norm{z_{k+1} - z_k}^2 - \Psi(z_{k+1}) \right) \\
&\leq (\alpha_{k+1}^2 L_{k+1}) (\phi(x_{k+1}) - \phi(y_{k+1})) - (\alpha_{k+1}^2 L_{k+1}) \Psi(x_{k+1}) \\
& + (\alpha_{k+1}^2 L_{k+1} - \alpha_{k+1}) \Psi(y_k).
\end{aligned}
\end{equation}
Finally, combining the previous estimates, we get
\begin{equation}
\begin{aligned}
\alpha_{k+1} \langle \nabla f(x_{k+1}), z_k - u \rangle &\leq (\alpha_{k+1}^2 L_{k+1}) (\phi(x_{k+1}) - \phi(y_{k+1})) \\
&\quad  + (V_{z_k}(u) - V_{z_{k+1}}(u)) - (\alpha_{k+1}^2 L_{k+1}) \Psi(x_{k+1}) \\
& \quad + (\alpha_{k+1}^2 L_{k+1} - \alpha_{k+1}) \Psi(y_k) + \alpha_{k+1} \Psi(u).
\end{aligned}
\end{equation}
\qed
\end{proof}

\begin{lemma}\label{lem:second_lemma}
For any $u \in Q$ and $\tau_k = \frac{1}{\alpha_{k+1} L_{k+1}}$ we have
%\begin{equation}\label{eq:lemma2}
%(\alpha_{k+1}^2 L_{k+1}) \phi(y_{k+1}) - (\alpha_{k+1}^2 L_{k+1} - \alpha_{k+1}) \phi(y_k) + (V_{z_{k+1}}(u) - V_{z_k}(u)) \leq \alpha_{k+1} \phi(u).
%\end{equation}
%and
\begin{equation}\label{eq:lemma2PD}
\begin{aligned}
&(\alpha_{k+1}^2 L_{k+1}) \phi(y_{k+1}) - (\alpha_{k+1}^2 L_{k+1} - \alpha_{k+1}) \phi(y_k)  \\
& \leq \alpha_{k+1} \left(f(x_{k+1}) + \langle \nabla f(x_{k+1}), u - x_{k+1}\rangle + \Psi(u) \right)+ (V_{z_{k}}(u) - V_{z_{k+1}}(u)).
\end{aligned}
\end{equation}
\end{lemma}
\begin{proof}
Using convexity of $f$ and relation $\tau_k (x_{k+1} - z_k) = (1 - \tau_k) (y_k - x_{k+1})$, we obtain
\begin{equation}
\begin{aligned}
%&\alpha_{k+1} (\phi(x_{k+1}) - \phi(u)) \\
%&\quad= \alpha_{k+1} (\Psi(x_{k+1}) - \Psi(u)) + \alpha_{k+1} ( f(x_{k+1}) - f(u) ) \\
&\alpha_{k+1} (\Psi(x_{k+1}) - \Psi(u)) + \alpha_{k+1} \langle \nabla f(x_{k+1}), x_{k+1} - u \rangle \\
&\quad= \alpha_{k+1} (\Psi(x_{k+1}) - \Psi(u)) + \alpha_{k+1} \langle \nabla f(x_{k+1}), x_{k+1} - z_k \rangle \\
&\quad+ \alpha_{k+1} \langle \nabla f(x_{k+1}), z_k - u \rangle \\
&\quad\leq \alpha_{k+1} (\Psi(x_{k+1}) - \Psi(u)) + \frac{\alpha_{k+1} (1 - \tau_k)}{\tau_k} \langle \nabla f(x_{k+1}), y_k - x_{k+1} \rangle \\
&\quad+ \alpha_{k+1} \langle \nabla f(x_{k+1}), z_k - u \rangle \\
&\quad\leq \alpha_{k+1} (\Psi(x_{k+1}) - \Psi(u)) + (\alpha_{k+1}^2 L_{k+1} - \alpha_{k+1}) (f(y_k) - f(x_{k+1})) \\
&\quad+ \alpha_{k+1} \langle \nabla f(x_{k+1}), z_k - u \rangle \\
&\quad\leq \alpha_{k+1} \phi(x_{k+1}) - \alpha_{k+1} \Psi(u) + (\alpha_{k+1}^2 L_{k+1} - \alpha_{k+1}) f(y_k) \\
&\quad - (\alpha_{k+1}^2 L_{k+1}) f(x_{k+1}) + \alpha_{k+1} \langle \nabla f(x_{k+1}), z_k - u \rangle. \\
\end{aligned}
\end{equation}
Now we apply Lemma~\ref{lem:first_lemma} to bound the last term, group the terms and get
\begin{equation}
\begin{aligned}
&\alpha_{k+1} (\Psi(x_{k+1}) - \Psi(u)) + \alpha_{k+1} \langle \nabla f(x_{k+1}), x_{k+1} - u \rangle \\
&\quad \leq \alpha_{k+1} \phi(x_{k+1}) - (\alpha_{k+1}^2 L_{k+1}) \phi(y_{k+1})  \\
&\quad + (\alpha_{k+1}^2 L_{k+1} - \alpha_{k+1}) \phi(y_k)+ (V_{z_k}(u) - V_{z_{k+1}}(u)).
\end{aligned}
\end{equation}
After rearranging, we obtain~\eqref{eq:lemma2PD}.
\qed
\end{proof}

Now we are ready to prove the convergence theorem for Algorithm~\ref{alg:main}.
\begin{theorem}
For the sequence $\{y_k\}_{k \geq 0}$ in Algorithm~\ref{alg:main} we have 
\begin{equation}\label{eq:PD_rate}
(\alpha_T^2 L_T) \phi(y_T) \leq \min_{x \in Q} \left\{\sum_{k=1}^T \alpha_k \left(f(x_{k}) + \langle \nabla f(x_{k}), u - x_{k}\rangle + \Psi(u) \right) + V_{z_{0}}(u) \right\} 
\end{equation}
and, hence, the following rate of convergence:
\begin{equation}\label{eq:rate}
\phi(y_T) - \phi(x^*) \leq \frac{4 L_f R^2}{T^2}.
\end{equation}
\end{theorem}
\begin{proof}
Note that the special choice of $\{\alpha_k\}_{k \geq 0}$ in Algorithm~\ref{alg:main} gives us
\begin{equation}\label{eq:alpha_relation}
\alpha_{k+1}^2 L_{k+1} - \alpha_{k+1} = \alpha_k^2 L_k, \qquad k \geq 0.
\end{equation}
Therefore, taking the sum over $k=0, \dots, T-1$ in \eqref{eq:lemma2PD} and using that $\alpha_0 = 0$, $V_{z_T}(u) \geq 0$ we get, for any $u \in Q$,
\begin{equation}
(\alpha_T^2 L_T) \phi(y_T) \leq \sum_{k=1}^T \alpha_k \left(f(x_{k}) + \langle \nabla f(x_{k}), u - x_{k}\rangle + \Psi(u) \right) + V_{z_{0}}(u) 
\end{equation}
and \eqref{eq:PD_rate} is straightforward.
At the same time, using the convexity of $f(x)$, the definition of $\phi(x)$, and $u = x^* = \argmin_{x \in Q} \phi(x)$, we obtain
\begin{equation}
\begin{aligned}
&(\alpha_T^2 L_T) \phi(y_T) \leq \min_{x \in Q} \left\{\sum_{k=1}^T \alpha_k \left(f(x_{k}) + \langle \nabla f(x_{k}), u - x_{k}\rangle + \Psi(u) \right) + V_{z_{0}}(u) \right\} \\
& \leq \left( \sum_{k=1}^T \alpha_k \right) \phi(x^*) + V_{z_0}(x^*).
\end{aligned}
\end{equation}
From~\eqref{eq:alpha_relation} it follows that $\sum_{k=1}^T \alpha_k = \alpha_T^2 L_T$, so
\begin{equation}\label{eq:final_bound}
\phi(y_T) \leq \phi(x^*) + \frac{1}{\alpha_T^2 L_T} V_{z_0}(x^*).
\end{equation}
Now it remains to estimate the rate of growth of coefficients $A_k := \alpha_k^2 L_k$. For this we use the technique from~\cite{Nes2013}. Note that from~\eqref{eq:alpha_relation} we have
\begin{equation}
A_{k+1} - A_k = \sqrt{\frac{A_{k+1}}{L_{k+1}}}
\end{equation}
Rearranging and using $(a+b)^2 \leq 2a^2 + 2b^2$ and $A_k \leq A_{k+1}$, we get
\begin{equation}
\begin{aligned}
A_{k+1} &= L_{k+1} (A_{k+1} - A_k)^2 = L_{k+1} \left( \sqrt{A_{k+1}} + \sqrt{A}_k \right)^2 \left( \sqrt{A_{k+1}} - \sqrt{A}_k \right)^2 \\
&\leq 4 L_{k+1} A_{k+1} \left( \sqrt{A_{k+1}} - \sqrt{A}_k \right)^2
\end{aligned}
\end{equation}
From this it follows that
\begin{equation}
\sqrt{A_{k+1}} \geq \frac{1}{2} \sum_{i=0}^k \frac{1}{\sqrt{L_i}}.
\end{equation}
Note that according to~\eqref{eq:lipschitz} and the stopping criterion for choosing $L_{k+1}$ in Algorithm~\eqref{alg:main}, we always have $L_i \leq 2 L_f$. Hence,
\begin{equation}\label{eq:A_growth}
\sqrt{A_{k+1}} \geq \frac{k+1}{2 \sqrt{2 L_f}} \qquad \Longleftrightarrow \qquad A_{k+1} \geq \frac{(k+1)^2}{8 L_f}.
\end{equation}

Thus, combining~\eqref{eq:A_growth} and \eqref{eq:final_bound} with $V_{z_0}(x^*) =: \frac{R^2}{2}$, we have proved \eqref{eq:rate}.
\qed
\end{proof}

Using the same arguments to~\cite{Nes2013}, it is also possible to prove that the average number of evaluations of the function $f$ per iteration in Algorithm~\ref{alg:main} equals 4.
\begin{theorem}
Let $N_k$ be the total number of evaluations of the function $f$ in Algorithm~\ref{alg:main} after the first $k$ iterations. Then for any $k \geq 0$ we have
\begin{equation}
N_k \leq 4 (k+1) + 2 \log_2 \frac{L_f}{L_0}.
\end{equation}
\end{theorem}

\section{Application to the Equilibrium Problem}
In this section, we apply Algorithm~\ref{alg:main} to solve the dual problem in \eqref{Gas_eq2} 
$$
\mathop {\min}\limits_{t\in {\rm
dom}\,\sigma ^\ast } \Bigl\{ {\gamma ^1\psi
^1(t/\gamma^1)+\sum\limits_{e\in E} {\sigma _e^\ast (t_e)} }
\Bigr\}.
$$ 
with $t$ in the role of $x$, $\gamma ^1\psi^1(t/\gamma^1)$ in the role of $f(x)$, and $\sum\limits_{e\in E} {\sigma _e^\ast (t_e)}$ in the role of $\Psi(x)$.

The inequality \eqref{eq:PD_rate} leads to the fact that Algorithm~\ref{alg:main} is primal-dual
\cite{Nes2009,NemOnnRoth2010,GasDvuEtc2015,GasGasEtc2015}, which
means that the sequences $\{t^i\}$ (which is in the role of $\{x_k\}$) and $\{\tilde {t}^i\}$ (which is in the role of $\{y_k\}$)
generated by this method have the following property:
\[
\gamma ^1\psi ^1(\tilde {t}^T/\gamma ^1) + \sum\limits_{e\in E}
{\sigma _e^\ast (\tilde {t}_e^T)} 
\]
\begin{equation}
\label{Gas_eq3}  - \mathop {\min }\limits_{t\in{\rm dom}\,\sigma^\ast}
\left\{ \frac{1}{A_T } {\sum\limits_{i=0}^T {\left[\alpha_i ( {\gamma
^1\psi^1(t^i/\gamma^1) + \langle {\nabla \psi
^1(t^i/\gamma^1),t-t^i}\rangle })\right]} } + \sum\limits_{e\in E}
{\sigma_e^\ast(t_e)} \right\} 
\end{equation}
\[
\leq \frac{4L_2 R_2^2 }{T^2 },
\]
where 
\[
L_2 \le (1/{\mathop {\min }\limits_{k=1,...,m} \gamma
^k})\sum\limits_{w^1\in OD^1} {d_{w^1}^1 } \cdot (l_{w^1})^2,
\]
with $l_{w^1}$ being the total number of edges (among all of the levels) in
the longest path for correspondence $w^1$,
\[
R_2^2 =\max \{\tilde {R}_2^2 ,\hat {R}_2^2\}, \quad \tilde {R}_2^2 =
(1/2)\left\| {\bar {t}-t^\ast } \right\|_2^2 , \quad \hat {R}_2^2 =
(1/2)\sum\limits_{e\in E} {\left( {\tau _e \left( {\bar {f}_e^N }
\right)-t_e^\ast } \right)^2} ,
\]
$\bar{f}^N$ is defined in Theorem 2, the method starts from $t^0
= \bar{t}$, $t^\ast$ is a solution of the problem (\ref{Gas_eq2}).

\begin{theorem} 
Let the problem (\ref{Gas_eq2}) be solved
by Algorithm~\ref{alg:main} generating sequences $\{t^i\}, \{\tilde
{t}^i\}$. Then. after $T$ iterations one has
\[
0\le \Bigl\{ {\gamma ^1\psi ^1(\tilde {t}^T/\gamma ^1) +
\sum\limits_{e\in E} {\sigma _e^\ast (\tilde{t}_e^T)} } \Bigr\} +
\Psi (\bar {x}^T,\bar {f}^T) \le \frac{4L_2 R_2^2 }{T^2 },
\]
\textit{where}
\[
f^i=\Theta x^i=-\nabla \psi ^1(t^i/\gamma ^1), \quad x^i=\bigl\{
{x_{p^k}^{k,i} } \bigr\}_{p^k\in P_{w^k}^k ,w^k\in OD^k}^{k=1,...,m}
,
\]
\[
x_{p^k}^{k,i} =d_{w^k}^k \frac{\exp (-g_{p^k}^k (t^i)/\gamma^k)}
{\sum\limits_{\tilde {p}^k\in P_{w^k}^k } \exp (-g_{\tilde {p}^k}^k
(t^i)/\gamma ^k) }, \quad p^k\in P_{w^k}^k , \quad w^k\in OD^k,
\quad k=1,...,m,
\]
\[
\bar {f}^T=\frac{1}{A_T }\sum\limits_{i=0}^T {\alpha_i f^i} , \quad \bar
{x}^T=\frac{1}{A_T }\sum\limits_{i=0}^T {\alpha_i x^i} .
\]
\end{theorem}

Theorem 2 provides the bound for the number of 
iterations in order to solve the problem (\ref{Gas_eq2}) with given accuracy. 
Nevertheless, on each iteration it is necessary to calculate
$\nabla \psi ^1(t/\gamma^1)$ and also $\psi
^1(t/\gamma^1)$. Similarly to
\cite{GasGasEtc2015,Gas2013,Nes2007} it is possible to show, using
the smoothed version of Bellman--Ford method, that for this
purpose it is enough to perform ${\rm O}({|{O^1}| |E|\mathop {\max
}\limits_{w^1\in OD^1} l_{w^1}})$ arithmetic operations.

In general, it is worth noting that the approach of adding some artificial 
vertices, edges, sources, sinks is very useful in different applications
\cite{Gas2015,BabGasLagMen2015,VasGasEtc2014}. 
%In particular, introduction
%fictitious (with zero expenses) paths-edges which give the chance of
%nothing to do to users (not to move \cite{Gas2015}, not to trade
%\cite{VasGasEtc2014}, etc.) is rather popular that, in turn, allows
%to consider situations with unstable correspondences $d^1$
%\cite{Gas2015,VasGasEtc2014}.
%
%Also popular method is transferring of costs of overcoming of
%vertices (knots) of the graph (cross-roads \cite{Gas2015}, sorting
%yards \cite{VasGasEtc2014}) in the costs of passing of additional
%edges which appeared when ``disentangling'' knots.
%
%But perhaps the most important for the majority of applications is
%introduction of a fictitious general source and general sink,
%connected by additional edges to already available vertices the
%graph \cite{Gas2015,BabGasLagMen2015,VasGasEtc2014}.

\textbf{Acknowledgements.} The research was supported by RFBR,
research project No. 15-31-70001 mol\_a\_mos and No. 15-31-20571
mol\_a\_ved.


\end{document}